\documentclass[11pt]{article}

\usepackage[margin=1in]{geometry}
\usepackage{amsmath,amssymb,amsthm,mathtools}
\usepackage{enumitem}
\usepackage[colorlinks=true,linkcolor=blue,citecolor=blue,urlcolor=blue]{hyperref}
\usepackage[normalem]{ulem}

\newtheorem{theorem}{Theorem}[section]
\newtheorem{lemma}[theorem]{Lemma}
\newtheorem{proposition}[theorem]{Proposition}

\newtheorem{claim}[theorem]{Claim}
\theoremstyle{definition}

\newtheorem{remark}[theorem]{Remark}
\newtheorem{problem}[theorem]{Problem}

\newcommand{\eps}{\varepsilon}

\newcommand{\redsout}[1]{%
  \bgroup
  \markoverwith{\textcolor{red}{\rule[0.5ex]{2pt}{0.4pt}}}%
  \ULon{#1}%
}

\newcommand{\rmynote}[1]{\textcolor{blue}{[RMY: #1]}}

\title{Exact Homomorphism Thresholds Beyond Cliques}
\author{Xinqi Huang\thanks{School of Mathematical Sciences, University of Science and Technology of China, Hefei, China and Extremal Combinatorics and Probability Group (ECOPRO), Institute for Basic Science (IBS), Daejeon, South Korea.
Email:huangxq@mail.ustc.edu.cn. Xinqi Huang was supported by the USTC Excellent PhD Students Overseas, the Institute for Basic Science (IBS-R029-C4), the National Key Research and Development Programs of China 2023YFA1010200, the NSFC under Grants No. 12171452 and No. 12231014 and Innovation Program for Quantum Science and Technology 2021ZD0302902.
}
\and
Mingyuan Rong\thanks{School of Mathematical Sciences, USTC, Hefei, China. Supported by National Key R\&D Program of China 2023YFA1010201, NSFC Grant No. 12125106 and the USTC Excellent PhD Students Overseas Study Program. Email:\texttt{rong\_ming\_yuan@mail.ustc.edu.cn}}
\and Chong Shangguan\thanks{Research Center for Mathematics and Interdisciplinary Sciences, Shandong University, Qingdao 266237, China, and the Frontiers Science Center for Nonlinear Expectations, Ministry of Education, Qingdao 266237, China. Supported by the National Natural Science Foundation of China under Grant Nos. 12571352 and 12231014, and the Fundamental Research Funds for the Central Universities. Email: theoreming@163.com.}
}
\date{}

\begin{document}
\maketitle

\begin{abstract}
The chromatic threshold, originating in a question of Erd\H{o}s and
Simonovits, asks when a linear minimum-degree condition forces bounded
chromatic number in $H$-free graphs. Motivated by a question of Thomassen,
the homomorphism threshold asks for the stronger conclusion that every such
graph admits a homomorphism to an $H$-free graph of bounded order. Since the
work of Goddard and Lyle determined the clique case, exact homomorphism
thresholds for individual non-complete forbidden graphs have remained
unknown.
In this paper, we extend the clique case to a larger family of forbidden
graphs, determining the homomorphism threshold exactly for every graph in
this family.
\end{abstract}

\section{Introduction}

A central theme in extremal graph theory is that an $H$-free graph with
sufficiently large minimum degree is often forced to exhibit strong global
structure and, consequently, bounded complexity. Over the past decades,
minimum-degree thresholds have provided a natural framework for quantifying
how much density is required to guarantee such structure.

The modern study of these threshold phenomena began with the chromatic
threshold problem, proposed by Erd\H{o}s and
Simonovits~\cite{1973ErdosSimonovits}. They introduced what is now known as
the \emph{chromatic threshold} of a graph $H$:
\begin{align*}
\delta_\chi(H):=\inf \{\alpha\ge 0: \exists~C=C(H, d) \textup{ s.t. $\forall~H$-free $G$ with }\delta(G)\ge \alpha |V(G)|
\Rightarrow \text{$\chi(G) \leq C$}\}.
\end{align*}
Following a series of works~\cite{2011Unpubilished,1973ErdosSimonovits,
2011JGTKrChromatic,1982OddCycles,1995DMJin,2010ColoringViaVCDim,
2010arxivKrfree,2002Thomassen,2007OddCycleChromatic},
Allen, B\"ottcher, Griffiths, Kohayakawa, and
Morris~\cite{2013advAllChromatic} completely determined the chromatic
threshold of every graph $H$. In particular, for each fixed chromatic number
$r\geq 3$, the chromatic threshold of an $r$-chromatic graph can take only
three possible values.

Having bounded chromatic number is equivalent to admitting a homomorphism to
a complete graph of bounded order. Motivated by this observation,
Thomassen~\cite{2002Thomassen} proposed a natural strengthening of the
chromatic-threshold problem in which the bounded homomorphic image is itself
required to be $H$-free. Recall that
$G\xrightarrow{\textup{hom}}F$ means that there is an
adjacency-preserving map from $V(G)$ to $V(F)$. The \emph{homomorphism
threshold} of a graph $H$ is defined by
\begin{align*}
    \delta_{\textup{hom}}(H):=
\inf\Bigl\{\alpha\ge 0:\exists\ H\textup{-free }F=F(H,\alpha)\
\textup{s.t. $\forall~H$-free $G$ with }\delta(G)\ge \alpha |V(G)|
\Rightarrow G\xrightarrow{\textup{hom}}F\Bigr\}
\end{align*}
Clearly, $\delta_{\textup{hom}}(H)\ge \delta_\chi(H)$.

The first major results showed that equality holds for cliques.
{\L}uczak~\cite{2006CombTriangleHom} determined the triangle case, and
Goddard and Lyle~\cite{2011JGTKrChromatic} proved that, for every $r\ge 3$,
$\delta_{\textup{hom}}(K_r)
    =
    \delta_\chi(K_r)
    =
    \frac{2r-5}{2r-3}$.
Beyond cliques, however, homomorphism thresholds appear to be much more
difficult to determine. This difficulty is already visible in another
fundamental test case: odd cycles. Ebsen and
Schacht~\cite{2020COMBHomoOddCycle} proved the general upper bound
$\delta_{\textup{hom}}(C_{2k-1})\le 1/(2k-1)$ for every $k\ge 3$.
Unexpectedly, Sankar~\cite{2022Maya} used topological methods to prove that
$\delta_{\textup{hom}}(C_{2k-1})>0$ for every $k\ge 3$. 
This stands in
stark contrast to Thomassen's theorem that
$\delta_\chi(C_{2k-1})=0$ for every $k\ge 3$. 
Beyond these results, very little is known about homomorphism thresholds.
In particular, prior to the present work, exact values had been determined
only for cliques; even for odd cycles, only positivity and an upper bound
were known.

In this paper, we extend the clique case to a broader family of forbidden
graphs. For integers $r\ge 3$, $k\ge 1$, and $1\le s\le r$, let
$T_{r,k}^{s}$ be the graph obtained from $k$ copies of $K_r$ by identifying
a common clique of order $s$ and keeping all remaining vertices disjoint.
For example, $T_{3,k}^{1}$ consists of $k$ triangles sharing a single common
vertex, while $T_{4,k}^{2}$ consists of $k$ copies of $K_4$ sharing a common
edge. In the special case $s=1$, we write $T_{r,k}$ for $T_{r,k}^{1}$.
Notice also that if $k=1$ or $s=r$, then $T_{r,k}^{s}=K_r$.

We completely determine the homomorphism threshold of $T_{r,k}^{s}$ for all
possible choices of $r$, $k$, and $s$.

\begin{theorem}\label{thm:hom-clique-fan}
For integers $r\ge 3$, $k\ge 1$, and $1\le s\le r$, we have
$\delta_{\textup{hom}}(T_{r,k}^{s})=\frac{2r-5}{2r-3}$.
\end{theorem}

To the best of our knowledge, Theorem~\ref{thm:hom-clique-fan} provides the
first infinite family of non-complete forbidden graphs whose homomorphism
thresholds are determined exactly. In particular, it shows that the
homomorphism-threshold value for a clique remains unchanged when a single
copy of $K_r$ is replaced by several copies of $K_r$ glued along a common.

We conclude this section with a brief overview of the proof.
The lower bound follows directly from the classification of chromatic
thresholds, which gives~\cite{2013advAllChromatic}
$\delta_{\textup{hom}}(T_{r,k}^{s})
\ge
\delta_\chi(T_{r,k}^{s})
=
\frac{2r-5}{2r-3}$.
The main difficulty is therefore to establish the matching upper bound.
When $s\ge 2$, our clique-counting lemma implies that every
$T_{r,k}^{s}$-free graph above the required minimum-degree threshold is
already $K_r$-free. The result then follows from the known structural theorem
for dense $K_r$-free graphs~\cite{2024GraphToGeom}.

The case $s=1$ requires an additional argument.
We prove the following stronger blow-up statement: for every
$\varepsilon>0$, there is a constant $C=C(r,k,s,\varepsilon)$ such that every
maximal $T_{r,k}^{s}$-free graph $G$ on $n$ vertices satisfying
$\delta(G)\ge
\left(\frac{2r-5}{2r-3}+\varepsilon\right)n$
is a blow-up of a $T_{r,k}^{s}$-free graph on at most $C$ vertices.
This statement immediately implies the desired upper bound. Indeed, every
$T_{r,k}^{s}$-free graph can be extended, on the same vertex set, to a
maximal $T_{r,k}^{s}$-free graph without decreasing its minimum degree. A
homomorphism from this maximal extension to a bounded
$T_{r,k}^{s}$-free graph then restricts to a homomorphism from the original
graph.

The main technical ingredient is a common-neighborhood clique-counting lemma
(Lemma~\ref{lemma:low level contains high level}), which strengthens the
dense-edge phenomenon of Fox and Wigderson~\cite{fox2023minimum}. Roughly
speaking, the lemma asserts that if the common neighborhood of a tuple
contains many copies of a smaller clique, then deleting one vertex from the
tuple produces a common neighborhood containing many copies of a larger
clique.

To prove the case $s=1$, we call an edge \emph{heavy} if its common
neighborhood contains many copies of $K_{r-2}$. The graph consisting of the
heavy edges has bounded matching number and hence admits a vertex cover of
bounded size. After deleting this vertex cover, we obtain a $K_r$-free
graph. We then apply the clique homomorphism-threshold theorem both to the
remaining graph and to certain neighborhood subgraphs. Finally, further
matching and vertex-cover arguments allow us to refine the resulting bounded
partition until every pair of parts is either complete or empty. This yields
a bounded blow-up quotient, and a lifting argument shows that the quotient is
itself $T_{r,k}$-free.


\section{Preliminaries and the clique-counting tool}\label{sec:tools}

Here we first list various useful tools. For graphs $H$ and $G$, define
\[
C(H,G):=\bigl|\{\,S\subseteq V(G):\ G[S]\cong H\,\}\bigr|,
\]
that is, the number of vertex-subsets inducing a copy of $H$. For a set
$W\subseteq V(G)$, let $N(W)$ denote its common neighborhood, with the
convention that $N(\emptyset)=V(G)$. The first lemma is based on a simple
pigeonhole-principle argument.

\begin{lemma}\label{lemma:largedegree many common neighbours}
Let $1\le p\le q$ be integers. Let $G$ be an $n$-vertex graph with
$\delta(G)\ge \delta n$. Assume that
$\delta>\frac{p-1}{q}$ and
\(
\bigl(\delta-\frac{p-1}{q}\bigr)n>2(q-p).
\)
Then, for every set $V=\{v_1,\dots,v_q\}\subseteq V(G)$ of $q$ distinct
vertices, there exists a subset $V_0\subseteq V$ with $|V_0|=p$ such that
\(
|N(V_0)|\ge \alpha n,
\)
where
\(
\alpha:=\frac{q\delta-p+1}{2\binom{q}{p}(q-p+1)}>0.
\)
\end{lemma}

\begin{proof}[Proof of Lemma~\ref{lemma:largedegree many common neighbours}]
Fix $V=\{v_1,\dots,v_q\}\subseteq V(G)$ and write
$U:=V(G)\setminus V$. Assume for contradiction that
$|N(S)|<\alpha n$ for every $p$-subset $S\subseteq V$. In particular,
\begin{equation}\label{eq:star}
|N(S)\cap U|<\alpha n
\qquad\text{for all }S\in\binom{V}{p}.
\end{equation}
Define
\(
W:=\{(i,u):\ i\in[q],\ u\in U,\ v_i u\in E(G)\}.
\)
Then
\begin{equation}\label{eq:W-lower}
|W|=\sum_{i=1}^q |N(v_i)\cap U|
\ge \sum_{i=1}^q\bigl(\delta n-(q-1)\bigr)
= q\delta n-q(q-1).
\end{equation}
For $u\in U$, let $s(u):=|N(u)\cap V|$. Then
$|W|=\sum_{u\in U}s(u)$. Let
\(
M:=\bigl|\{u\in U:\ s(u)\ge p\}\bigr|.
\)
Consider the set of pairs
\[
\mathcal{P}:=\{(S,u):\ S\in\binom{V}{p},\ u\in U,\ S\subseteq N(u)\}.
\]
By \eqref{eq:star}, each $S\in\binom{V}{p}$ has fewer than $\alpha n$
common neighbors in $U$, and hence
\(
|\mathcal{P}|<\binom{q}{p}\alpha n.
\)
On the other hand, every $u\in U$ with $s(u)\ge p$ contributes at least one
pair $(S,u)\in\mathcal{P}$, so $|\mathcal{P}|\ge M$. Therefore,
\(
M<\binom{q}{p}\alpha n.
\)

To upper-bound $|W|=\sum_{u\in U}s(u)$, notice that vertices with
$s(u)\le p-1$ contribute at most $p-1$ each, while vertices with
$s(u)\ge p$ contribute at most $q$ each. Therefore,
\begin{equation}\label{eq:W-upper}
|W|\le (p-1)(|U|-M)+qM
\le (p-1)(n-q)+(q-p+1)M
< (p-1)(n-q)+(q-p+1)\binom{q}{p}\alpha n.
\end{equation}
Combining \eqref{eq:W-lower} and \eqref{eq:W-upper} yields
\begin{equation}\label{eq:key}
\Bigl(q\delta-(p-1)-(q-p+1)\binom{q}{p}\alpha\Bigr)n
<q(q-p).
\end{equation}
By the definition of $\alpha$, the coefficient on the left-hand side of
\eqref{eq:key} equals
\[
q\delta-(p-1)-\frac{q\delta-p+1}{2}
=\frac{q\delta-p+1}{2}.
\]
Hence \eqref{eq:key} implies
\(
\frac{q\delta-p+1}{2}\,n<q(q-p),
\)
contradicting
\[
(q\delta-p+1)n
=q\Bigl(\delta-\frac{p-1}{q}\Bigr)n
>2q(q-p).
\]
This completes the proof.
\end{proof}

We next generalize the preceding result.

\begin{lemma}\label{lemma:low level contains high level}
Let \(r\), \(k\), and \(t\) be positive integers satisfying \(r\ge 3\) and
\(k+2t\le 2r-3\). For every \(\varepsilon>0\) and \(\alpha>0\), there exist
\(n_0=n_0(r,\varepsilon,\alpha)\) and
\(\alpha'=\alpha'(r,\varepsilon,\alpha)>0\) such that the following holds.
Let \(G\) be an \(n\)-vertex graph with \(n\ge n_0\) and
\(
\delta(G)\ge \bigl(\frac{2r-5}{2r-3}+\varepsilon\bigr)n.
\)
Suppose that there exists a subset
\(W=\{w_1,\dots,w_k\}\subseteq V(G)\) such that
$C(K_t,G[N(W)])\ge \alpha n^t$.

Then there exists a subset \(W'\subseteq W\) with \(|W'|=k-1\) such that
$C(K_{t+1},G[N(W')])\ge \alpha' n^{t+1}$.
\end{lemma}

\begin{proof}[Proof of Lemma~\ref{lemma:low level contains high level}]
Let \(\mathcal T\) be the family of all \(t\)-subsets of \(V(G)\) spanning
copies of \(K_t\) in \(G[N(W)]\). Then
\[
|\mathcal T|=C(K_t,G[N(W)])\ge \alpha n^t.
\]

Set
\(
d:=\frac{2r-5}{2r-3}+\varepsilon.
\)
We first fix a constant \(\gamma=\gamma(r,\varepsilon)>0\). Apply
Lemma~\ref{lemma:largedegree many common neighbours} with
\(
q:=k+2t
\)
and
\(
p:=q-1=k+2t-1.
\)
Indeed,
\[
\frac{p-1}{q}
=\frac{k+2t-2}{k+2t}
\le \frac{2r-5}{2r-3}
<d,
\]
where the first inequality follows from \(k+2t\le 2r-3\). Moreover, since
\(d-\frac{p-1}{q}\ge\varepsilon\), the remaining numerical hypothesis of
Lemma~\ref{lemma:largedegree many common neighbours} holds whenever \(n\)
is sufficiently large in terms of \(r\) and \(\varepsilon\). The constant
provided by that lemma is
\(
\frac14\bigl(d-\frac{q-2}{q}\bigr).
\)
Since \(3\le q\le 2r-3\), we may choose \(\gamma>0\) uniformly in \(k\)
and \(t\), depending only on \(r\) and \(\varepsilon\). It follows that
every \(q\)-set of vertices contains a \((q-1)\)-subset with at least
\(\gamma n\) common neighbors.

Set
\(
\beta:=\frac{\gamma}{2}
\)
and define
$\mathcal T_1
:=
\{\,T\in\mathcal T:\ |N(W\cup T)|\ge \beta n\,\}$.
For each \(T\in\mathcal T_1\), every vertex in \(N(W\cup T)\) extends
\(T\) to a copy of \(K_{t+1}\) inside \(G[N(W)]\). Thus, by double
counting pairs \((T,v)\) with \(T\in\mathcal T_1\) and
\(v\in N(W\cup T)\), we obtain
\begin{equation}\label{eq:easy-ineq-Kt1}
C(K_{t+1},G[N(W)])
\ge \frac{\beta n\,|\mathcal T_1|}{t+1}.
\end{equation}

If
\(
|\mathcal T_1|\ge \frac{|\mathcal T|}{2}
\ge \frac{\alpha}{2}n^t,
\)
then \eqref{eq:easy-ineq-Kt1} gives
\[
C(K_{t+1},G[N(W)])
\ge \frac{\beta\alpha}{2(t+1)}\,n^{t+1}.
\]
Since \(N(W)\subseteq N(W')\) for every \(W'\subseteq W\), any subset
\(W'\subseteq W\) of size \(k-1\) satisfies
\[
C(K_{t+1},G[N(W')])
\ge C(K_{t+1},G[N(W)])
\ge \frac{\beta\alpha}{2(t+1)}\,n^{t+1}.
\]
Thus the conclusion holds in this case.

It remains to consider the case
\(
|\mathcal T_1|<\frac{|\mathcal T|}{2}.
\)
Choose
\(
T_0\in\mathcal T\setminus\mathcal T_1.
\)
Then
\(
|N(W\cup T_0)|<\beta n.
\)
Let
$\mathcal T_2
:=
\{\,T\in\mathcal T:\ T\cap T_0=\emptyset\,\}$.
The number of \(t\)-subsets of \(V(G)\) intersecting the fixed \(t\)-set
\(T_0\) is at most
\(
t\binom{n}{t-1}\le t n^{t-1}.
\)
Consequently, provided that \(n\ge 2t/\alpha\), we have
\begin{equation}\label{eq:T2-lower}
|\mathcal T_2|
\ge |\mathcal T|-t n^{t-1}
\ge \alpha n^t-t n^{t-1}
\ge \frac{\alpha}{2}n^t.
\end{equation}

Fix \(T\in\mathcal T_2\), and consider
\(
V:=W\cup T_0\cup T.
\)
Since \(T_0,T\subseteq N(W)\) and \(T_0\cap T=\emptyset\), these three
sets are pairwise disjoint, and hence
\(
|V|=k+2t=q.
\)
By the choice of \(\gamma\), there exists a subset \(V_0\subseteq V\) with
\(
|V_0|=q-1=k+2t-1
\)
and
\(
|N(V_0)|\ge \gamma n.
\)

We claim that
\(
T\subseteq V_0.
\)
Indeed, if \(T\nsubseteq V_0\), then, since \(|V\setminus V_0|=1\), the
unique omitted vertex lies in \(T\). It follows that
\(
W\cup T_0\subseteq V_0.
\)
Therefore
\(
N(V_0)\subseteq N(W\cup T_0),
\)
and hence
\[
|N(V_0)|
\le |N(W\cup T_0)|
<\beta n
=\frac{\gamma}{2}n,
\]
contradicting \(|N(V_0)|\ge\gamma n\). Thus \(T\subseteq V_0\), as
claimed.

Since \(|V\setminus V_0|=1\) and \(T\subseteq V_0\), the omitted vertex
lies in \(W\cup T_0\). In particular,
\(
|V_0\cap W|\ge k-1.
\)
Choose a subset
\(
W_T\subseteq V_0\cap W
\)
with
\(
|W_T|=k-1.
\)
Then
\(
W_T\cup T\subseteq V_0,
\)
so
\(
N(V_0)\subseteq N(W_T\cup T),
\)
and therefore
\begin{equation}\label{eq:WTT-large}
|N(W_T\cup T)|
\ge |N(V_0)|
\ge \gamma n.
\end{equation}

There are exactly \(\binom{k}{k-1}=k\) possible subsets of \(W\) of
order \(k-1\). By the pigeonhole principle, there exists a fixed subset
\(
W'\subseteq W
\)
with \(|W'|=k-1\) such that \(W_T=W'\) for at least
\(|\mathcal T_2|/k\) sets \(T\in\mathcal T_2\). For every such \(T\),
\eqref{eq:WTT-large} gives
\(
|N(W'\cup T)|\ge\gamma n.
\)
Every vertex \(v\in N(W'\cup T)\) extends \(T\) to a copy of \(K_{t+1}\)
in \(G[N(W')]\). Therefore, by double counting,
\[
C(K_{t+1},G[N(W')])
\ge \frac{1}{t+1}\cdot\frac{|\mathcal T_2|}{k}\cdot\gamma n.
\]
Using \eqref{eq:T2-lower}, we obtain
\[
C(K_{t+1},G[N(W')])
\ge \frac{1}{t+1}\cdot\frac{\alpha}{2k}n^t\cdot\gamma n
=\frac{\alpha\gamma}{2k(t+1)}\,n^{t+1}.
\]

Since \(t+1\le r\) and \(k\le 2r-5\), both cases imply the conclusion
with
$\alpha'
:=
\min\left\{
\frac{\alpha\beta}{2r},
\frac{\alpha\gamma}{2r(2r-5)}
\right\}>0$.
This constant depends only on \(r\), \(\varepsilon\), and \(\alpha\).
This completes the proof.
\end{proof}

We also have the following easy corollary of
Lemma~\ref{lemma:low level contains high level}.

\begin{lemma}\label{lemma:quadraticCliques}
Let $r,\ell,t$ be integers satisfying $r\ge 3$, $1\le t\le\ell$, and
$\ell+t\le 2r-2$. For every $\varepsilon>0$, there exist
$n_0=n_0(r,\varepsilon)$ and $\alpha=\alpha(r,\varepsilon)>0$ such that
the following holds. Let $G$ be an $n$-vertex graph with $n\ge n_0$ and
\(
\delta(G)\ge \bigl(\frac{2r-5}{2r-3}+\varepsilon\bigr)n.
\)

Then, for any set $W=\{w_1,\ldots,w_\ell\}\subseteq V(G)$, there exists
a subset $W'\subseteq W$ with $|W'|=\ell-t$ such that
$C\bigl(K_t,G[N(W')]\bigr)\ge \alpha n^t$.
\end{lemma}

\begin{proof}[Proof of Lemma~\ref{lemma:quadraticCliques}]
We argue by induction on \(t\).

Suppose first that \(t=1\). If \(\ell=1\), take \(W'=\emptyset\). By the
convention \(N(\emptyset)=V(G)\), we have
\(
C(K_1,G[N(W')])=n,
\)
so the conclusion holds with \(\alpha_1=1\).

We may therefore assume that \(\ell\ge 2\). Since
\(\ell+1\le 2r-2\), we have \(\ell\le 2r-3\), and hence
\(
\frac{\ell-2}{\ell}\le \frac{2r-5}{2r-3}
<\frac{\delta(G)}{n}.
\)
Therefore Lemma~\ref{lemma:largedegree many common neighbours}, applied
with
\(
q=\ell
\)
and
\(
p=\ell-1,
\)
yields a subset \(W_1\subseteq W\) of size \(\ell-1\) such that
\(
|N(W_1)|\ge\alpha_1n
\)
for some constant \(\alpha_1=\alpha_1(r,\varepsilon)>0\). Since
\(
C(K_1,G[N(W_1)])=|N(W_1)|,
\)
the desired conclusion follows in the base case.

Now assume that \(t\ge2\), and suppose that the statement has already
been proved for \(t-1\). Applying the induction hypothesis to the same
\(\ell\)-set \(W\), we obtain a subset
\(
W_{t-1}\subseteq W
\)
with
\(
|W_{t-1}|=\ell-(t-1)=\ell-t+1
\)
such that
$C\bigl(K_{t-1},G[N(W_{t-1})]\bigr)
\ge \alpha_{t-1}n^{t-1}$
for some constant \(\alpha_{t-1}=\alpha_{t-1}(r,\varepsilon)>0\).

We now apply Lemma~\ref{lemma:low level contains high level} with
\(
k:=|W_{t-1}|=\ell-t+1
\)
and with \(t-1\) in place of \(t\). The required numerical condition is
\(
k+2(t-1)\le 2r-3.
\)
Indeed,
\(
k+2(t-1)=(\ell-t+1)+2t-2=\ell+t-1\le 2r-3,
\)
where the last inequality follows from \(\ell+t\le2r-2\). Hence
Lemma~\ref{lemma:low level contains high level} gives a subset
\(
W_t\subseteq W_{t-1}
\)
with
\(
|W_t|=k-1=\ell-t
\)
such that
$C\bigl(K_t,G[N(W_t)]\bigr)\ge \alpha_t n^t$,
where \(\alpha_t=\alpha_t(r,\varepsilon,\alpha_{t-1})>0\).

Since \(t\) is fixed and each constant \(\alpha_j\) depends only on
\(r\) and \(\varepsilon\), we may set
\[
\alpha:=\min\{\alpha_1,\alpha_2,\dots,\alpha_t\}>0.
\]
After taking the minimum over the finitely many admissible choices of
\(\ell\) and \(t\), if necessary, the constant \(\alpha\) may be chosen
to depend only on \(r\) and \(\varepsilon\). The subset \(W':=W_t\)
then satisfies
\(
|W'|=\ell-t
\)
and
\(
C\bigl(K_t,G[N(W')]\bigr)\ge\alpha n^t.
\)
This completes the proof.
\end{proof}

\begin{remark}
For the special choice $\ell=r$ and $t=r-2$,
Lemma~\ref{lemma:quadraticCliques} recovers the dense-edge phenomenon
used by Fox and Wigderson~\cite{fox2023minimum}. Indeed, given any copy
of $K_r$ in $G$, the lemma gives a $2$-subset
$W'\subseteq V(K_r)$ such that
$C\bigl(K_{r-2},G[N(W')]\bigr)\ge \alpha n^{r-2}$.
Since $W'$ is an edge of this copy of $K_r$, this edge lies in at least
$\alpha n^{r-2}$ copies of $K_r$.
\end{remark}

More importantly, Lemma~\ref{lemma:low level contains high level} has a
crucial inductive form: it turns many smaller cliques in the common
neighborhood of a tuple into many larger cliques after discarding one
vertex from the tuple. Iterating this step allows us to pass between
clique extensions of different sizes and common cores of different
orders, which is essential for handling several cliques sharing a common
core. In this sense, the lemma is the key clique-amplification tool in
the paper and may be of independent interest.


We next record two standard external tools that will be used in the proof
of the main theorem. The first is the known homomorphism-threshold upper
bound and blow-up theorem for cliques.

\begin{theorem}[\cite{2024GraphToGeom}]\label{thm:KrHomo}
For every $r\ge3$ and $\varepsilon>0$, every $n$-vertex $K_r$-free graph
$G$ satisfying
$\delta(G)\ge
\left(\frac{2r-5}{2r-3}+\varepsilon\right)n$
is homomorphic to a $K_r$-free graph $H$ with
$|V(H)|\le 2^{(1/\varepsilon)^C}$ for some constant $C=C(r)>0$.
Moreover, if $G$ is maximal $K_r$-free, then $H$ may be chosen so that
$G$ is a blow-up of $H$.
\end{theorem}

The second is K{\"o}nig's theorem. Recall that a \emph{vertex cover} in a
graph is a set of vertices containing at least one endpoint of every edge,
while a \emph{matching} is a set of pairwise disjoint edges. K{\"o}nig's
theorem identifies these two parameters in bipartite graphs.

\begin{theorem}[K{\"o}nig's theorem~\cite{1931Konig}]\label{thm:Konig}
In every bipartite graph, the cardinality of a maximum matching equals the
cardinality of a minimum vertex cover.
\end{theorem}

\section{Proof of Theorem~\ref{thm:hom-clique-fan}}\label{sec:proof-of-main-result}

In this section we prove Theorem~\ref{thm:hom-clique-fan}. Throughout the
section, we assume that $G$ is a maximal $T_{r,k}^{s}$-free graph on $n$
vertices satisfying
    $\delta(G)\ge \left(\frac{2r-5}{2r-3}+\eps\right)n$,
where $\eps>0$ is fixed and $n\ge n_0(\eps,r)$ is sufficiently large. 
If $k=1$ or $s=r$, then $T_{r,k}^s=K_r$, and the result follows from
the clique case. Henceforth we may assume $k\ge 2$ and $s<r$.

We first deal with the easier case $s\ge 2$.

\begin{proof}[Proof of Theorem~\ref{thm:hom-clique-fan} when $s \ge 2$]
    We claim the following.
    \begin{claim}\label{claim:reduce-to-clique}
        $G$ is $K_r$-free.
    \end{claim}
    Assuming Claim~\ref{claim:reduce-to-clique}, we may apply
    Theorem~\ref{thm:KrHomo} to conclude that $G$ is homomorphic to a
    $K_r$-free graph of bounded order, where the bound depends only on
    $\eps$ and $r$. Since every $K_r$-free graph is $T_{r,k}^{s}$-free, this
    gives the desired bounded $T_{r,k}^{s}$-free homomorphic image.

    It remains to prove Claim~\ref{claim:reduce-to-clique}.
    \begin{proof}[Proof of Claim~\ref{claim:reduce-to-clique}]
        Suppose, for a contradiction, that $G$ contains a copy of $K_r$ with
        vertex set $\{v_1,\ldots,v_r\}$. By
        Lemma~\ref{lemma:quadraticCliques}, applied with parameters
        $\ell=r$ and $t=r-s$, there exists a subset
        $W=\{v_1',\ldots,v_s'\}\subseteq \{v_1,\ldots,v_r\}$ spanning a copy
        of $K_s$ such that
            $C(K_{r-s},G[N(W)])\ge \alpha n^{r-s}$.
        Since $n$ is sufficiently large, this guarantees at least $k$
        vertex-disjoint copies of $K_{r-s}$ in $G[N(W)]$. Together with the
        common clique $W$, these copies form a copy of $T_{r,k}^{s}$, a
        contradiction. This proves the claim.
    \end{proof}
\end{proof}

It remains to prove the case $s=1$.
The proof of the case when $r = 3$ is simple, we leave the proofs to the appendix.
In the remaining section, we always assume that $r\ge 4$.


\subsection{Proof of Theorem~\ref{thm:hom-clique-fan} when \(s=1\)}
\label{subsec:s-one}

In this subsection, we prove Theorem~\ref{thm:hom-clique-fan} when
\(s=1\), which is the main technical case. Recall that
\(T_{r,k}:=T_{r,k}^{1}\). The case \(r=3\) is treated separately in
the appendix, so throughout this subsection we assume that \(r\ge4\).
Fix \(\varepsilon>0\), and let \(G\) be a maximal \(T_{r,k}\)-free
graph on \(n\) vertices satisfying
$\delta(G)\ge
\left(\frac{2r-5}{2r-3}+\varepsilon\right)n$.
It suffices to prove that \(G\) is a blow-up of a bounded
\(T_{r,k}\)-free graph. Indeed, any \(T_{r,k}\)-free graph can be
extended on the same vertex set to a maximal \(T_{r,k}\)-free graph,
and a homomorphism of the latter restricts to a homomorphism of the
original graph.

We first give a brief outline. We identify all heavy edges and show
that all their endpoints lie in a bounded set \(F\). After deleting
\(F\), the remaining graph is \(K_r\)-free, so the clique
homomorphism theorem gives a bounded partition into independent sets.
We refine this partition according to the neighborhoods of vertices
in \(F\). The main remaining step is to show that whenever two parts
have at least one edge between them, their bipartite complement has
bounded matching number. A final application of K{\"o}nig's theorem
then turns the partition into a bounded blow-up partition.

\paragraph{Choice of the heavy-edge constant.}
We first choose all constants that will be used in the definition of a
heavy edge. Recall that
\(
\theta_j:=\frac{2j-5}{2j-3}
\)
for \(j\ge 3\). Let \(\rho_0=\rho_0(r,\varepsilon)>0\) be the constant
given by Lemma~\ref{lemma:quadraticCliques}, applied with
\(\ell=r\) and \(t=r-2\). Thus, every \(r\)-set \(X\subseteq V(G)\)
contains a pair \(\{x,y\}\subseteq X\) such that
$C\bigl(K_{r-2},G[N_G(x,y)]\bigr)\ge \rho_0n^{r-2}$.

When \(r\ge 5\), we define one further sequence of constants. Fix a
vertex \(c\in V(G)\), put \(m_c:=|N_G(c)|\), and let
\(G_c:=G[N_G(c)]\). We may assume that
\(\theta_r+\varepsilon<1\), since otherwise no such graph \(G\) exists
for all sufficiently large \(n\). For every \(u\in N_G(c)\),
\[
d_{G_c}(u)
\ge d_G(u)-(n-m_c)
\ge m_c+(\theta_r+\varepsilon-1)n.
\]
Since \(m_c\ge(\theta_r+\varepsilon)n\), it follows that
\[
\frac{d_{G_c}(u)}{m_c}
\ge 2-\frac{1}{\theta_r+\varepsilon}
=\theta_{r-1}+\frac{\varepsilon}{\theta_r(\theta_r+\varepsilon)}
\ge \theta_{r-1}+\frac{\varepsilon}{2}.
\]
Apply Lemma~\ref{lemma:quadraticCliques} inside \(G_c\), with \(r-1\)
in place of \(r\), \(\ell=2r-5\), \(t=1\), and excess
\(\varepsilon/2\). There exists \(\sigma=\sigma(r,\varepsilon)>0\)
such that every \((2r-5)\)-set \(P\subseteq N_G(c)\) contains a vertex
\(u\in P\) satisfying
\[
|N_{G_c}(P\setminus\{u\})|
\ge \sigma m_c
\ge \lambda_1n,
\qquad
\lambda_1:=\sigma\theta_r>0.
\]
For \(j=1,\dots,r-3\), recursively apply
Lemma~\ref{lemma:low level contains high level} with
\[
k_j:=r-j,\qquad t_j:=j,\qquad \alpha:=\lambda_j,
\]
and let \(\lambda_{j+1}>0\) be a corresponding output constant. These
applications are legitimate because
\[
k_j+2t_j=(r-j)+2j=r+j\le 2r-3
\qquad\text{for }1\le j\le r-3.
\]

We now define
\[
\alpha_{\mathrm h}:=
\begin{cases}
\rho_0, & r=4,\\[2mm]
\min\{\rho_0,\lambda_{r-2}\}, & r\ge 5.
\end{cases}
\]
An edge \(xy\in E(G)\) is called \emph{heavy} if
$C\bigl(K_{r-2},G[N_G(x,y)]\bigr)
\ge \alpha_{\mathrm h}n^{r-2}$.
In particular, every copy of \(K_r\) contains a heavy edge, since
\(\alpha_{\mathrm h}\le\rho_0\). All constants above depend only on
\(r\) and \(\varepsilon\) and are chosen before \(F\), the partition
\(A_1,\dots,A_M\), and the matching parameter \(t\) are defined.

\begin{proposition}\label{prop:finite-upset}
Let \(F\) be the union of the endpoints of all heavy edges in \(G\).
Then the following hold.
\begin{enumerate}[label=(\roman*)]
    \item\label{item:|F|-small}
    \(
    |F|\le 4k^2/\alpha_{\mathrm h}.
    \)

    \item\label{item:G-F-free}
    The graph \(G[V(G)\setminus F]\) is \(K_r\)-free.
\end{enumerate}
\end{proposition}

\begin{proof}[Proof of Proposition~\ref{prop:finite-upset}]
Let \(J\) be the spanning subgraph of \(G\) whose edge set consists
of all heavy edges. We first claim that
\begin{equation}\label{eq:heavy-max-degree}
\Delta(J)\le k-1.
\end{equation}
Otherwise, some vertex \(v\) has \(k\) distinct neighbors
\(u_1,\dots,u_k\) in \(J\). For every \(i\in[k]\), the graph
\(G[N_G(v,u_i)]\) contains at least
\(\alpha_{\mathrm h}n^{r-2}\) copies of \(K_{r-2}\). Since \(r\) and
\(k\) are fixed, we may greedily choose copies
\[
R_i\cong K_{r-2}\subseteq G[N_G(v,u_i)]
\qquad (i\in[k])
\]
such that the sets \(V(R_i)\), \(i\in[k]\), are pairwise disjoint
and avoid \(\{v,u_1,\dots,u_k\}\). Indeed, at each step only
\(O_{r,k}(1)\) vertices are forbidden, and the number of
\((r-2)\)-cliques meeting them is \(O_{r,k}(n^{r-3})\). The \(k\)
copies of \(K_r\) on
\(\{v,u_i\}\cup V(R_i)\), \(i\in[k]\), intersect exactly in \(v\),
and hence form a copy of \(T_{r,k}\), a contradiction. This proves
\eqref{eq:heavy-max-degree}.

We next bound the matching number of \(J\). Suppose that
$\{x_iy_i:i\in[m]\}$
is a matching in \(J\), and set \(U_i:=N_G(x_i,y_i)\). 
For \(u\in U_i\), let
$f_i(u):=
C\bigl(K_{r-3},G[U_i\cap N_G(u)]\bigr)$.
Every copy of \(K_{r-2}\) in \(G[U_i]\) is counted once at each of
its \(r-2\) vertices, so
\[
\sum_{u\in U_i}f_i(u)
=(r-2)C(K_{r-2},G[U_i])
\ge(r-2)\alpha_{\mathrm h}n^{r-2}.
\]
Define
$C_i:=
\left\{u\in U_i:
f_i(u)\ge\frac{\alpha_{\mathrm h}}2n^{r-3}\right\}$.
Since \(f_i(u)\le n^{r-3}\), the preceding inequality implies
\begin{equation}\label{eq:Ci-lower}
|C_i|\ge\frac{\alpha_{\mathrm h}}2n.
\end{equation}
Indeed, if \(|C_i|<\alpha_{\mathrm h}n/2\), then
\[
\sum_{u\in U_i}f_i(u)
<\frac{\alpha_{\mathrm h}}2n\cdot n^{r-3}
+n\cdot\frac{\alpha_{\mathrm h}}2n^{r-3}
=\alpha_{\mathrm h}n^{r-2},
\]
contradicting \(r\ge4\).

If \(m>2k/\alpha_{\mathrm h}\), then
\eqref{eq:Ci-lower} and the pigeonhole principle give a vertex
\(w\in V(G)\) belonging to at least \(k\) of the sets \(C_i\).
Choose distinct indices \(i_1,\dots,i_k\) with
\(w\in C_{i_j}\). For every $j\in[k]$, the graph
$G[U_{i_j}\cap N_G(w)]$ contains at least
$(\alpha_{\mathrm h}/2)n^{r-3}$ copies of $K_{r-3}$.
Since $r$ and $k$ are fixed, we may greedily choose pairwise
vertex-disjoint copies
$L_j\cong K_{r-3}\subseteq G[U_{i_j}\cap N_G(w)]$
for all $j\in[k]$
such that every $L_j$ avoids $w$ and the $2k$ vertices
$x_{i_1},y_{i_1},\ldots,x_{i_k},y_{i_k}$.
Indeed, at each step only $O_{r,k}(1)$ vertices are forbidden,
whereas the number of copies of $K_{r-3}$ meeting those vertices
is $O_{r,k}(n^{r-4})$.
Then the $k$ copies of $K_r$ on
$\{w,x_{i_j},y_{i_j}\}\cup V(L_j)$
where $j\in[k]$
intersect exactly in $w$, again giving a copy of $T_{r,k}$.
Consequently,
\begin{equation}\label{eq:heavy-matching}
\nu(J)\le\frac{2k}{\alpha_{\mathrm h}}.
\end{equation}

Let \(\mathcal M\) be a maximal matching in \(J\). Every vertex of
positive degree in \(J\) either belongs to \(V(\mathcal M)\) or is
adjacent to a vertex of \(V(\mathcal M)\). Hence
\[
|F|
\le 2|\mathcal M|\bigl(\Delta(J)+1\bigr)
\le
2\cdot\frac{2k}{\alpha_{\mathrm h}}\cdot k
=\frac{4k^2}{\alpha_{\mathrm h}},
\]
proving \ref{item:|F|-small}.

Finally, every copy of \(K_r\) in \(G\) contains a heavy edge, and
both endpoints of every heavy edge belong to \(F\). Thus no copy of
\(K_r\) is contained in \(V(G)\setminus F\), proving
\ref{item:G-F-free}.
\end{proof}

We next perform the following preprocessing step.

\begin{proposition}\label{prop:preprocess}
There exists a partition
\(
V(G)\setminus F=A_1\sqcup\cdots\sqcup A_M
\)
such that the following hold.
\begin{enumerate}[label=(\roman*)]
    \item\label{item:independent-sets}
    Each \(A_i\) is an independent set in \(G\).

    \item\label{item:0-1-partition}
    For every \(i\in[M]\) and every \(v\in F\), either
    \(A_i\subseteq N_G(v)\) or
    \(A_i\cap N_G(v)=\emptyset\).

    \item\label{item:number-of-parts}
    \(M\le 2^{(1/\varepsilon)^m}\) for some constant
    \(m=m(r,k)>0\).

    \item\label{item:total-Kr-free}
    Let \(\Gamma\) be the auxiliary graph with vertex set \([M]\),
    where \(ij\in E(\Gamma)\) if and only if
    \(G[A_i,A_j]\) is non-empty. Then \(\Gamma\) is \(K_r\)-free.

    \item\label{item:neighbor-Kr-1-free}
    For each \(v\in F\), let \(\Gamma_v\) be the auxiliary graph
    with
    \[
    V(\Gamma_v)=\{i\in[M]:A_i\subseteq N_G(v)\},
    \]
    where \(ij\in E(\Gamma_v)\) if and only if
    \(G[A_i,A_j]\) is non-empty. Then \(\Gamma_v\) is
    \(K_{r-1}\)-free.
\end{enumerate}
\end{proposition}

\begin{proof}[Proof of Proposition~\ref{prop:preprocess}]
Let \(G_1:=G[V(G)\setminus F]\). By
Proposition~\ref{prop:finite-upset}\ref{item:G-F-free}, the graph
\(G_1\) is \(K_r\)-free. Moreover, by
Proposition~\ref{prop:finite-upset}\ref{item:|F|-small}, for all
sufficiently large \(n\),
\[
\delta(G_1)
\ge(\theta_r+\varepsilon)n-|F|
\ge\left(\theta_r+\frac{\varepsilon}{2}\right)|V(G_1)|.
\]
By Theorem~\ref{thm:KrHomo}, there exist a \(K_r\)-free graph \(Q\)
and a homomorphism \(\varphi:V(G_1)\to V(Q)\) such that
\begin{equation}\label{eq:q-bound}
q:=|V(Q)|\le 2^{(1/\varepsilon)^{C_1}}
\end{equation}
for some \(C_1=C_1(r)>0\). In particular, for every \(z\in V(Q)\),
the preimage \(\varphi^{-1}(z)\) is an independent set in \(G_1\).

Fix \(v\in F\), and set \(m_v:=|N_{G_1}(v)|\). We claim that
\(G_1[N_{G_1}(v)]\) is \(K_{r-1}\)-free. Otherwise, a copy of
\(K_{r-1}\) in \(N_{G_1}(v)\), together with \(v\), would form a
copy of \(K_r\) in \(G\). Every copy of \(K_r\) contains a heavy
edge, whose two endpoints lie in \(F\), but this copy of \(K_r\)
contains only one vertex of \(F\), namely \(v\), a contradiction.

For every \(u\in N_{G_1}(v)\), we have
\begin{align}\label{eq:deg-in-neigh}
d_{G_1[N_{G_1}(v)]}(u)
\ge d_G(u)-(n-m_v)
\ge m_v+(\theta_r+\varepsilon-1)n.
\end{align}
Also
\(
m_v\ge d_G(v)-|F|
\ge(\theta_r+\varepsilon/2)n
\)
for all sufficiently large \(n\). It follows from
\eqref{eq:deg-in-neigh} that
\begin{equation}\label{eq:mindeg-neigh}
\delta\bigl(G_1[N_{G_1}(v)]\bigr)
\ge
\left(\theta_{r-1}+\frac{\varepsilon}{2}\right)m_v.
\end{equation}
Indeed,
\[
1-\frac{1-\theta_r-\varepsilon}{\theta_r+\varepsilon/2}
\ge\theta_{r-1}+\frac{\varepsilon}{2}.
\]

Applying Theorem~\ref{thm:KrHomo} to
\(G_1[N_{G_1}(v)]\), with \(r-1\) in place of \(r\) and
\(\varepsilon/2\) as the excess parameter, we obtain a
\(K_{r-1}\)-free graph \(Q^v\) and a homomorphism
$\psi_v:N_{G_1}(v)\to V(Q^v)$
such that
\begin{equation}\label{eq:qv-bound}
q_v:=|V(Q^v)|\le 2^{(1/\varepsilon)^{C_2}}
\end{equation}
for some \(C_2=C_2(r)>0\).

Extend \(\psi_v\) to \(V(G_1)\) by introducing an additional symbol
\(0\notin V(Q^v)\) and setting \(\psi_v(x)=0\) whenever
\(x\notin N_{G_1}(v)\). Define an equivalence relation on
\(V(G_1)\) by
\[
x\sim y
\quad\Longleftrightarrow\quad
\varphi(x)=\varphi(y)
\ \text{and}\
\psi_v(x)=\psi_v(y)\text{ for every }v\in F.
\]
Let \(A_1,\dots,A_M\) be the equivalence classes.

Since \(\varphi\) is constant on every \(A_i\), each \(A_i\) is an
independent set, proving \ref{item:independent-sets}. For every
\(v\in F\), the function \(\psi_v\) is constant on each \(A_i\).
If its common value is \(0\), then
\(A_i\cap N_G(v)=\emptyset\); otherwise
\(A_i\subseteq N_G(v)\). This proves
\ref{item:0-1-partition}.

Suppose that \(\{i_1,\dots,i_r\}\) spans a copy of \(K_r\) in
\(\Gamma\). Let \(z_a\) be the common \(\varphi\)-image of \(A_{i_a}\).
For distinct \(a,b\in[r]\), the pair
\(G[A_{i_a},A_{i_b}]\) contains an edge, and hence
\(z_az_b\in E(Q)\). Thus \(Q\) contains a copy of \(K_r\), a
contradiction. This proves \ref{item:total-Kr-free}.

Similarly, fix \(v\in F\) and suppose that
\(\{i_1,\dots,i_{r-1}\}\) spans a copy of \(K_{r-1}\) in
\(\Gamma_v\). Let \(z_a^v\) be the common \(\psi_v\)-image of
\(A_{i_a}\). The non-empty pairs between these classes imply that
\(z_a^vz_b^v\in E(Q^v)\) for all distinct \(a,b\), so \(Q^v\)
contains a copy of \(K_{r-1}\), a contradiction. This proves
\ref{item:neighbor-Kr-1-free}.

Finally, every equivalence class is determined by one
\(\varphi\)-value and, for each \(v\in F\), one of the \(q_v+1\)
possible \(\psi_v\)-values. Hence
\begin{equation}\label{eq:M-bound}
M\le q\prod_{v\in F}(q_v+1)
\le 2^{(1/\varepsilon)^m}
\end{equation}
for some \(m=m(r,k)>0\). Here we use the quantitative bound
\(\alpha_{\mathrm h}^{-1}\le(1/\varepsilon)^{C_0}\), for some
\(C_0=C_0(r)>0\), which follows from the explicit constants in
Lemmas~\ref{lemma:largedegree many common neighbours},
\ref{lemma:low level contains high level}, and
\ref{lemma:quadraticCliques}. Together with
Proposition~\ref{prop:finite-upset}\ref{item:|F|-small}, this gives
\(|F|\le(1/\varepsilon)^{C_3}\) for some \(C_3=C_3(r,k)>0\).
The bounds \eqref{eq:q-bound} and \eqref{eq:qv-bound} now imply
\eqref{eq:M-bound}.
\end{proof}

Let
\begin{equation}\label{eq:t-def}
t:=10(M+|F|)^{r-2}(kr)^{kr+1},
\end{equation}
and let \(M_t\) denote a matching of size \(t\). For disjoint sets
\(X,Y\subseteq V(G)\), write
\(\overline{G[X,Y]}:=K_{X,Y}\setminus G[X,Y]\).

\begin{lemma}\label{lemma:MTTFree}
For any distinct \(p,q\in[M]\), if \(G[A_p,A_q]\) is non-empty,
then \(\overline{G[A_p,A_q]}\) is \(M_t\)-free.
\end{lemma}

We first complete the proof of Theorem~\ref{thm:hom-clique-fan}
assuming Lemma~\ref{lemma:MTTFree}.

\begin{proof}[Proof of Theorem~\ref{thm:hom-clique-fan} when \(s=1\)]
For every pair \(1\le p<q\le M\) for which
\(G[A_p,A_q]\) is non-empty, Lemma~\ref{lemma:MTTFree} and
K{\"o}nig's theorem give a vertex cover
$K_{pq}\subseteq A_p\cup A_q$
of \(\overline{G[A_p,A_q]}\) with
\(|K_{pq}|\le t-1\). 
Let
\begin{equation}\label{eq:K-bound}
K:=\bigcup_{\substack{1\le p<q\le M\\
G[A_p,A_q]\ne\emptyset}}K_{pq}\Rightarrow
|K|\le(t-1)\binom M2.
\end{equation}
For \(i\in[M]\), set \(B_i:=A_i\setminus K\). If
\(G[A_i,A_j]\) is empty, then so is \(G[B_i,B_j]\); otherwise
\(G[B_i,B_j]\) is complete by the definition of \(K_{ij}\).

Write \(K=\{u_1,\dots,u_{|K|}\}\). For \(i\in[M]\) and
\(J\subseteq[|K|]\), define
\[
B_i^J:=
\{x\in B_i:N_G(x)\cap K=\{u_j:j\in J\}\}.
\]
Discard the empty classes. Every \(B_i^J\) is independent. Between
two such classes all pairs are either complete or empty, and every
class is either complete or anticomplete to each singleton in
\(F\cup K\): for \(K\) this holds by definition, and for \(F\) it
follows from Proposition~\ref{prop:preprocess}
\ref{item:0-1-partition}. Therefore \(G\) is a blow-up of the
quotient graph \(L\) with vertex set
\[
V(L):=
F\sqcup K\sqcup
\{b_i^J:i\in[M],\ J\subseteq[|K|],\ B_i^J\ne\emptyset\},
\]
where adjacency is inherited from \(G\). Moreover,
\[
|V(L)|\le |F|+|K|+M2^{|K|},
\]
which is bounded in terms of \(r,k,\varepsilon\).

The graph \(L\) is \(T_{r,k}\)-free. Indeed, if \(L\) contained a
copy of \(T_{r,k}\), choosing one vertex from each corresponding
non-empty blow-up class would give a copy of \(T_{r,k}\) in \(G\),
a contradiction. Thus \(G\) is a blow-up of a bounded
\(T_{r,k}\)-free graph, completing the proof.
\end{proof}

\subsection{Proof of Lemma~\ref{lemma:MTTFree}}
In this subsection, we prove Lemma~\ref{lemma:MTTFree}, thereby
completing the proof.

\begin{proof}[Proof of Lemma~\ref{lemma:MTTFree}]
Fix distinct \(p,q\in[M]\) such that \(G[A_p,A_q]\) is non-empty.
Suppose for a contradiction that the bipartite complement
\(\overline{G[A_p,A_q]}\) contains a copy of \(M_t\), with
\[
\{a_1,\dots,a_t\}\subseteq A_p,
\qquad
\{b_1,\dots,b_t\}\subseteq A_q,
\]
where \(a_ib_i\notin E(G)\) for every \(i\in[t]\), and all these
\(2t\) vertices are distinct.

Write \(H:=T_{r,k}\). Since \(G\) is maximal \(H\)-free,
\(G+a_ib_i\) contains a copy \(\widehat H_i\) of \(H\) using the added
edge \(a_ib_i\). Let
\(
H_i^-:=\widehat H_i-a_ib_i\subseteq G.
\)
The edge \(a_ib_i\) lies in a unique blade of \(\widehat H_i\); let
\(T_i\) denote the corresponding incomplete blade in \(H_i^-\). Thus
$V(T_i)=\{a_i,b_i,z_1^i,\dots,z_{r-2}^i\}$,
and \(a_ib_i\) is the unique missing edge in \(G[V(T_i)]\). Let
\(c_i\in V(T_i)\) be the core of \(H_i^-\).

For \(h\in[r-2]\), define the \emph{location} of \(z_h^i\) to be the
vertex \(z_h^i\) itself if \(z_h^i\in F\), and the unique part \(A_s\)
containing \(z_h^i\) otherwise. There are at most \(M+|F|\) possible
locations. By the pigeonhole principle, there is a set
\(L\subseteq[t]\) such that
\begin{equation}\label{eq:L-size}
|L|\ge\frac{t}{(M+|F|)^{r-2}}
\end{equation}
and, for every \(h\in[r-2]\), the vertices \(z_h^i\), \(i\in L\),
have the same location. There are only \(r\) possible positions of
the core in \(T_i\), so there is \(L'\subseteq L\) such that
\begin{equation}\label{eq:Lprime-size}
|L'|\ge\frac{|L|}{r}
\end{equation}
and the core occupies the same position in \(T_i\) for every
\(i\in L'\). Relabel \(L'=\{1,\dots,\ell\}\). Then
\begin{equation}\label{eq:ell-lower}
\ell\ge\frac{t}{r(M+|F|)^{r-2}}\ge 2,
\end{equation}
where the last inequality follows from the definition of \(t\).
For \(h\in[r-2]\), set
\(
Z_h:=\{z_h^i:i\in[\ell]\}.
\)
Each \(Z_h\) is either contained in one part \(A_{s(h)}\), or is a
singleton \(\{w_h\}\) with \(w_h\in F\).

\begin{claim}\label{claim:two-singletons}
At least two of the sets \(Z_h\), \(h\in[r-2]\), are singletons.
\end{claim}

\begin{proof}[Proof of Claim~\ref{claim:two-singletons}]
Suppose first that none of the sets \(Z_h\) is a singleton. Then
\(Z_h\subseteq A_{s(h)}\) for every \(h\in[r-2]\). The \(r\) parts
$A_p,\ A_q,\ A_{s(1)},\dots,A_{s(r-2)}$
are pairwise distinct and every pair of them induces an edge. Indeed,
all required edges other than the one between \(A_p\) and \(A_q\)
occur in any fixed \(T_i\), while \(G[A_p,A_q]\) is non-empty by
assumption. They therefore span a copy of \(K_r\) in the auxiliary
graph \(\Gamma\), contradicting
\ref{item:total-Kr-free}.

Suppose next that exactly one \(Z_h\) is a singleton, say
\(Z_{h_0}=\{w\}\). If \(w\notin F\), let \(A_s\) be the part
containing \(w\). The same argument shows that
\[
A_p,\ A_q,\ A_s,\
\{A_{s(h)}:h\in[r-2]\setminus\{h_0\}\}
\]
span a copy of \(K_r\) in \(\Gamma\), again a contradiction.

It remains to consider \(w\in F\). For every
\(h\in[r-2]\setminus\{h_0\}\), the presence in \(T_i\) of the edges
\(wa_i\), \(wz_h^i\), \(a_iz_h^i\), and \(z_h^iz_{h'}^i\) implies,
using \ref{item:0-1-partition}, that
\[
A_p,\ A_q,\ \{A_{s(h)}:h\in[r-2]\setminus\{h_0\}\}
\]
are \(r-1\) pairwise distinct vertices of \(\Gamma_w\), every two of
which are adjacent. This gives a copy of \(K_{r-1}\) in
\(\Gamma_w\), contradicting \ref{item:neighbor-Kr-1-free}. The claim
follows.
\end{proof}

The locations of two different \(z\)-positions are distinct. Indeed,
the corresponding vertices are adjacent in each \(T_i\), whereas
every \(A_j\) is independent; moreover, two different positions
cannot be the same vertex of \(F\). Similarly, no \(z\)-position has
location \(A_p\) or \(A_q\). Consequently, for the two incomplete
blades \(T_1,T_2\), a vertex belongs to both blades precisely when it
occurs in the same \(z\)-position in both.

By Claim~\ref{claim:two-singletons}, after relabeling we may write
$V(T_1)\cap V(T_2)=W=\{w_1,\dots,w_{a+2}\}$
for some \(a\ge0\). Since the vertices \(a_1,b_1,a_2,b_2\) are all
distinct, \(|W|\le r-2\), and hence \(a\le r-4\). Put
\(
s:=r-a-4.
\)
After relabeling the remaining positions, we may write
\[
V(T_1)=
\{x_1,y_1,z_1^1,\dots,z_s^1,w_1,\dots,w_{a+2}\},
\]
\[
V(T_2)=
\{x_2,y_2,z_1^2,\dots,z_s^2,w_1,\dots,w_{a+2}\},
\]
where \(x_i=a_i\in A_p\), \(y_i=b_i\in A_q\), and \(x_iy_i\) is the
unique missing edge in \(T_i\). For every \(j\in[s]\), the distinct
vertices \(z_j^1,z_j^2\) lie in the same part of the partition and
are therefore non-adjacent. By the choice of \(L'\), the two cores
occupy corresponding positions: either
\[
c_1=x_1,\ c_2=x_2;\qquad
c_1=y_1,\ c_2=y_2;\qquad
c_1=z_j^1,\ c_2=z_j^2
\]
for some \(j\in[s]\), or \(c_1=c_2=w_h\) for some
\(h\in[a+2]\).

We shall repeatedly use the following two claims.

\begin{claim}\label{claim:Produce}
Let \(H_i^-\subseteq G\) have core \(c_i\) and incomplete blade
\(T_i\). Suppose that \(S\subseteq V(T_i)\) satisfies
\[
|S|=r-1,\qquad G[S]\cong K_{r-1},\qquad c_i\in S.
\]
If
\(
|N_G(S)|>|V(H_i^-)\cup F|,
\)
then \(G\) contains a copy of \(H\).
\end{claim}

\begin{proof}[Proof of Claim~\ref{claim:Produce}]
Choose
\(
u\in N_G(S)\setminus(V(H_i^-)\cup F).
\)
Then \(G[S\cup\{u\}]\cong K_r\). Since \(c_i\in S\), replacing the
incomplete blade \(T_i\) by this copy of \(K_r\), while leaving the
other \(k-1\) blades unchanged, gives a copy of \(H\) in \(G\).
\end{proof}

\begin{claim}\label{claim:FewKr-2}
For every \(i\in[t]\), there is a constant
\(C=C(r,k,\varepsilon)>0\) such that
\[
C\bigl(K_{r-2},G[N_G(a_i,b_i)]\bigr)\le Cn^{r-3}.
\]
\end{claim}

\begin{proof}[Proof of Claim~\ref{claim:FewKr-2}]
Fix \(i\in[t]\). We first show that every copy of \(K_{r-2}\) in
\(G[N_G(a_i,b_i)]\) meets \(F\). Otherwise, let \(K\) be such a copy
with \(V(K)\cap F=\emptyset\). Since the parts \(A_j\) are independent,
the vertices of \(K\) lie in \(r-2\) distinct parts
\(A_{s_1},\dots,A_{s_{r-2}}\), none of which is \(A_p\) or \(A_q\).
The parts
\[
A_p,\ A_q,\ A_{s_1},\dots,A_{s_{r-2}}
\]
are pairwise distinct and every pair of them induces an edge:
between \(A_p\) and \(A_q\) this follows from the assumption on
\(p,q\), and all remaining edges are witnessed by \(a_i,b_i\) and
the vertices of \(K\). They span a copy of \(K_r\) in \(\Gamma\),
contradicting \ref{item:total-Kr-free}.

Thus every such \(K_{r-2}\) contains a vertex of \(F\), and consequently
\[
C\bigl(K_{r-2},G[N_G(a_i,b_i)]\bigr)
\le |F|\binom{n}{r-3}
\le Cn^{r-3},
\]
because \(|F|=O_{r,k,\varepsilon}(1)\).
\end{proof}

\begin{claim}\label{claim:WhenA=0}
If \(a=0\), then \(w_1w_2\) is a heavy edge. Moreover, neither \(w_1\)
nor \(w_2\) is the core of \(H_1^-\) or \(H_2^-\).
\end{claim}

\begin{proof}[Proof of Claim~\ref{claim:WhenA=0}]
When \(a=0\), we have \(s=r-4\) and
\[
V(T_i)=
\{x_i,y_i,z_1^i,\dots,z_{r-4}^i,w_1,w_2\}
\qquad (i=1,2).
\]
The vertices \(x_i,y_i,z_1^i,\dots,z_{r-4}^i\) all lie outside \(F\):
this is clear for \(x_i,y_i\), and each pair \(z_j^1,z_j^2\) consists
of two distinct vertices in a common part \(A_{s(j)}\).

Apply Lemma~\ref{lemma:quadraticCliques} to \(V(T_i)\), with
\(\ell=r\) and \(t=r-2\). It gives a pair
\(\{u,v\}\subseteq V(T_i)\) such that
\[
C\bigl(K_{r-2},G[N_G(u,v)]\bigr)
\ge \rho_0n^{r-2}
\ge \alpha_{\mathrm h}n^{r-2}.
\]
If \(\{u,v\}\) is neither \(\{x_i,y_i\}\) nor
\(\{w_1,w_2\}\), then \(uv\in E(G)\), and at least one of \(u,v\)
lies outside \(F\). The displayed inequality would make \(uv\) a
heavy edge, contradicting the definition of \(F\). If
\(\{u,v\}=\{x_i,y_i\}\), the displayed inequality contradicts
Claim~\ref{claim:FewKr-2} for all sufficiently large \(n\). Hence
\(\{u,v\}=\{w_1,w_2\}\). Since \(w_1w_2\in E(G)\), this edge is
heavy.

Suppose now, by symmetry, that \(w_1\) is the core of \(H_i^-\) for
some \(i\in\{1,2\}\). The heaviness of \(w_1w_2\) gives at least
\(\alpha_{\mathrm h}n^{r-2}\) copies of \(K_{r-2}\) in
\(G[N_G(w_1,w_2)]\). The number of these copies meeting
\(V(H_i^-)\cup F\) is at most
\[
|V(H_i^-)\cup F|n^{r-3}
=O_{r,k,\varepsilon}(n^{r-3}).
\]
We may therefore choose
$R\cong K_{r-2}\subseteq G[N_G(w_1,w_2)]$
disjoint from \(V(H_i^-)\cup F\). Replacing \(T_i\) by the copy of
\(K_r\) on \(\{w_1,w_2\}\cup V(R)\) gives a copy of \(H\) in \(G\),
a contradiction.
\end{proof}

We now rule out every possible value of \(a\). The induction below is
taken within the class of configurations just described: the missing
endpoints lie in \(A_p,A_q\), corresponding non-shared \(z\)-vertices
lie in a common independent part, and the two core positions agree.

Suppose first that \(a=0\). By Claim~\ref{claim:WhenA=0}, the common
core position does not lie in \(W=\{w_1,w_2\}\). At least one of
\(x_1,y_1\) is therefore not the core of \(H_1^-\). Interchanging
\(p\) and \(q\) if necessary, we may assume that \(x_1\) is not the
core; then \(x_2\) is not the core of \(H_2^-\). Set
$P:=\bigl(V(T_1)\cup V(T_2)\bigr)\setminus\{x_1\}$.
Since \(V(T_1)\cap V(T_2)=\{w_1,w_2\}\), we have \(|P|=2r-3\).
Lemma~\ref{lemma:quadraticCliques}, applied with
\(\ell=2r-3\) and \(t=1\), gives a vertex \(u\in P\) such that
\begin{equation}\label{eq:a0-large-common}
|N_G(P\setminus\{u\})|\ge\beta n
\end{equation}
for some \(\beta=\beta(r,\varepsilon)>0\).

If \(u\in V(T_2)\setminus W\), then
\[
S_1:=V(T_1)\setminus\{x_1\}\subseteq P\setminus\{u\}.
\]
The set \(S_1\) spans \(K_{r-1}\) and contains the core of \(H_1^-\).
Thus \eqref{eq:a0-large-common} and Claim~\ref{claim:Produce} give a
copy of \(H\) for all sufficiently large \(n\), a contradiction.
Similarly, if \(u\in V(T_1)\setminus(\{x_1\}\cup W)\), use
\(
S_2:=V(T_2)\setminus\{x_2\}.
\)
It follows that \(u\in W\); by symmetry, assume that \(u=w_1\). Then
\begin{equation}\label{eq:a0-T2-common}
|N_G(V(T_2)\setminus\{w_1\})|\ge\beta n.
\end{equation}

The edge \(w_1w_2\) is heavy, so \(w_2\in F\). We claim that
$N_G(V(T_2)\setminus\{w_1\})\subseteq F$.
Otherwise choose
\(
v\in N_G(V(T_2)\setminus\{w_1\})\setminus F
\)
and let \(A(v)\) be its part. By
\ref{item:0-1-partition}, all the \(r-1\) parts
\[
A(v),\ A_p,\ A_q,\ A_{s(1)},\dots,A_{s(r-4)}
\]
are contained in \(N_G(w_2)\). They are pairwise distinct, and every
pair induces an edge: the edges incident with \(A(v)\) are witnessed
by \(v\), the edge between \(A_p,A_q\) exists by assumption, and all
other edges occur in \(T_2\). These parts therefore span a copy of
\(K_{r-1}\) in \(\Gamma_{w_2}\), contradicting
\ref{item:neighbor-Kr-1-free}. The claim and
\eqref{eq:a0-T2-common} now give
\(
\beta n\le |F|,
\)
which is impossible for all sufficiently large \(n\). This settles
the case \(a=0\).

Assume now that \(a\ge1\), and that all configurations with smaller
overlap parameter have already been ruled out. We first consider the
case in which the common core position does not lie in \(W\). For
\(i=1,2\), choose \(e_i\in\{x_i,y_i\}\) that is not the core of
\(H_i^-\), and set
\(
S_i:=V(T_i)\setminus\{e_i\}.
\)
Then \(G[S_i]\cong K_{r-1}\) and \(c_i\in S_i\).

Let \(P:=V(T_1)\cup V(T_2)\). Since
\(
|P|=2r-a-2\le2r-3,
\)
Lemma~\ref{lemma:quadraticCliques}, applied with
\(\ell=2r-a-2\) and \(t=1\), gives \(u\in P\) such that
\begin{equation}\label{eq:gen-a-large-common}
|N_G(P\setminus\{u\})|\ge\beta_a n
\end{equation}
for some \(\beta_a=\beta_a(r,\varepsilon)>0\).
If \(u\in V(T_1)\setminus W\), then
\(
S_2\subseteq P\setminus\{u\};
\)
if \(u\in V(T_2)\setminus W\), then
\(
S_1\subseteq P\setminus\{u\}.
\)
In either case Claim~\ref{claim:Produce} gives a copy of \(H\), a
contradiction.

We may therefore assume that \(u\in W\). Since the common core
position is outside \(W\), the vertex \(u\) is not the core of either
copy. By \eqref{eq:gen-a-large-common}, after deleting the bounded set
\(F\cup V(H_1^-)\cup V(H_2^-)\), linearly many vertices remain in
\(N_G(P\setminus\{u\})\). Since \(M\) is bounded independently of
\(n\), some part \(A_\sigma\) contains two distinct vertices
\[
v_1,v_2\in N_G(P\setminus\{u\})
\setminus\bigl(F\cup V(H_1^-)\cup V(H_2^-)\bigr).
\]
They are non-adjacent because \(A_\sigma\) is independent. Define
\[
T_1':=G[(V(T_1)\setminus\{u\})\cup\{v_1\}],
\qquad
T_2':=G[(V(T_2)\setminus\{u\})\cup\{v_2\}].
\]
Each \(T_i'\) is a copy of \(K_r^-\) with the same missing edge as
\(T_i\). Since \(u\) is not the core, replacing \(T_i\) by \(T_i'\)
in \(H_i^-\) gives a new copy \(H_i^{-\prime}\subseteq G\) with the
same core position. Moreover,
\[
V(T_1')\cap V(T_2')=W\setminus\{u\},
\]
and the new vertices \(v_1,v_2\) form a non-shared pair in the same
part \(A_\sigma\). Thus the new configuration has overlap parameter
\(a-1\), contradicting the induction hypothesis.

It remains to consider the case in which the common core belongs to
\(W\). Write \(c:=c_1=c_2=w_{a+2}\), let
\(m_c:=|N_G(c)|\), and put \(G_c:=G[N_G(c)]\). As in the choice of
the heavy-edge constant,
\begin{equation}\label{eq:Gc-min-degree}
\delta(G_c)\ge
\left(\theta_{r-1}+\frac{\varepsilon}{2}\right)m_c.
\end{equation}

Suppose first that \(a\ge2\), and set
$P_c:=\bigl(V(T_1)\cup V(T_2)\bigr)\setminus\{c\}$.
Then \(|P_c|=2r-a-3\) and
\(
|P_c|+1\le2(r-1)-2.
\)
By Lemma~\ref{lemma:quadraticCliques}, applied inside \(G_c\) with
\(r-1\) in place of \(r\), \(\ell=2r-a-3\), and \(t=1\), there is
\(u\in P_c\) such that
\begin{equation}\label{eq:Gc-large-common}
|N_{G_c}(P_c\setminus\{u\})|\ge\beta_a' n
\end{equation}
for some \(\beta_a'=\beta_a'(r,\varepsilon)>0\).

If \(u\in V(T_1)\setminus W\), choose
\(e_2\in\{x_2,y_2\}\) and set
\(
S_2^\circ:=V(T_2)\setminus\{c,e_2\}.
\)
Then \(S_2^\circ\subseteq P_c\setminus\{u\}\). The set
\(
S_2:=S_2^\circ\cup\{c\}
\)
spans a copy of \(K_{r-1}\), contains the core, and satisfies
\[
|N_G(S_2)|
=|N_{G_c}(S_2^\circ)|
\ge\beta_a'n.
\]
Claim~\ref{claim:Produce} gives a contradiction. The case
\(u\in V(T_2)\setminus W\) is symmetric.

Hence \(u\in W\setminus\{c\}\). As above, choose distinct vertices
\[
v_1,v_2\in A_\sigma\cap N_{G_c}(P_c\setminus\{u\})
\setminus\bigl(F\cup V(H_1^-)\cup V(H_2^-)\bigr)
\]
in a common part \(A_\sigma\). Replacing \(u\) in \(T_1,T_2\) by
\(v_1,v_2\), respectively, gives a configuration with overlap
parameter \(a-1\), contradicting the induction hypothesis.

Finally, suppose that \(a=1\). Then
\(
W=\{w_1,w_2,c\}
\)
and \(s=r-5\). Set
\[
P':=\bigl(V(T_1)\cup V(T_2)\bigr)\setminus\{c,x_1\}.
\]
The set \(P'\) is contained in \(N_G(c)\), has size \(2r-5\), and
therefore the choice of \(\sigma\) and \(\lambda_1\) gives a vertex
\(u\in P'\) such that
\begin{equation}\label{eq:a1-Gc-large-common}
|N_{G_c}(P'\setminus\{u\})|\ge\lambda_1n.
\end{equation}

If \(u\in V(T_1)\setminus(\{c,x_1\}\cup W)\), then
\[
S_2^\circ:=V(T_2)\setminus\{c,x_2\}
\subseteq P'\setminus\{u\}.
\]
If \(u\in V(T_2)\setminus W\), then
\[
S_1^\circ:=V(T_1)\setminus\{c,x_1\}
\subseteq P'\setminus\{u\}.
\]
In either case, \(S_i:=S_i^\circ\cup\{c\}\) is a copy of
\(K_{r-1}\) containing the core and has at least \(\lambda_1n\)
common neighbors. Claim~\ref{claim:Produce} yields a contradiction.
Hence \(u\in\{w_1,w_2\}\); by symmetry, assume \(u=w_2\).

Set
\[
P''':=V(T_2)\setminus\{w_2\}
=\{x_2,y_2,z_1^2,\dots,z_{r-5}^2,w_1,c\}.
\]
Since
\(
V(T_2)\setminus\{c,w_2\}\subseteq P'\setminus\{w_2\}
\)
and every vertex of \(N_{G_c}(P'\setminus\{w_2\})\) is adjacent to
\(c\), \eqref{eq:a1-Gc-large-common} implies
\[
C\bigl(K_1,G[N_G(P''')]\bigr)\ge\lambda_1n.
\]

Starting with \(W_1:=P'''\), apply
Lemma~\ref{lemma:low level contains high level} successively for
\(j=1,\dots,r-3\), using
\[
k'=r-j,\qquad t'=j,\qquad \alpha=\lambda_j.
\]
At step \(j\), the numerical condition is
\(
(r-j)+2j=r+j\le2r-3.
\)
By the recursive choice of the constants \(\lambda_j\), after the
last step we obtain a pair \(U=\{p',q'\}\subseteq P'''\) such that
\begin{equation}\label{eq:final-heavy-pair}
C\bigl(K_{r-2},G[N_G(p',q')]\bigr)
\ge\lambda_{r-2}n^{r-2}
\ge\alpha_{\mathrm h}n^{r-2}.
\end{equation}

The only non-edge in \(G[P''']\) is \(x_2y_2\). If
\(U=\{x_2,y_2\}\), then \eqref{eq:final-heavy-pair} contradicts
Claim~\ref{claim:FewKr-2} for all sufficiently large \(n\).
Therefore \(p'q'\in E(G)\), and \eqref{eq:final-heavy-pair} shows that
\(p'q'\) is heavy.

Every vertex of
\(
\{x_2,y_2,z_1^2,\dots,z_{r-5}^2\}
\)
lies outside \(F\). Since both endpoints of a heavy edge lie in \(F\),
the only remaining possibility is
\(
\{p',q'\}=\{w_1,c\}.
\)
The heaviness of \(w_1c\) gives
\(\alpha_{\mathrm h}n^{r-2}\) copies of \(K_{r-2}\) in
\(G[N_G(w_1,c)]\). For sufficiently large \(n\), one of these copies,
say \(R\), is disjoint from \(V(H_2^-)\cup F\). Replacing \(T_2\) by
the copy of \(K_r\) on \(\{w_1,c\}\cup V(R)\) gives a copy of \(H\)
in \(G\), the final contradiction.

This completes the induction and the proof of the lemma.
\end{proof}



\section{Concluding remarks}
In this paper we determine a larger family beyond cliques the exact values of homomorphism threshold.
The value $\frac{2r-5}{2r-3}$ is one of the three possible values of the chromatic threshold of a graph with chromatic number $r$.
So it is interesting to ask the following.
\begin{problem}
    Determine the family 
    $\{H: \delta_{\textup{hom}}(H) = \frac{2r-5}{2r-3}\}.$
\end{problem}

\paragraph{Acknowledgements.}
This work was initiated during our visit to the Institute for Basic Science (IBS) in 2024. We thank Hong Liu for his hospitality.

\bibliographystyle{abbrv}
\bibliography{Clique-fan}

\appendix
\section{Proof of Theorem~\ref{thm:hom-clique-fan} when $r=3$ and $s=1$}
\begin{proof}
Let $G$ be a maximal $T_{3,k}$-free graph with $\delta(G)\ge (1/3+\eps)n $ where $n = |V(G)|$.
We first isolate the edges which are forced by triangles. For every triangle $xyz$ in $G$, a simple inclusion-exclusion argument shows that one of the three pairs has at least $\eps n/3$ common neighbors. Indeed, if all three pairwise common neighborhoods had size less than $\eps n/3$, then
\[
        d(x)+d(y)+d(z)
        \le n+|N(x)\cap N(y)|+|N(x)\cap N(z)|+|N(y)\cap N(z)|
        <(1+\eps)n,
\]
contradicting $d(x)+d(y)+d(z)\ge (1+3\eps)n$.

Call an edge $uv$ \emph{heavy} if $u$ and $v$ lie in a triangle and
\[
        |N(u)\cap N(v)|\ge \frac{\eps n}{3}.
\]
Let $F$ be the subgraph of $G$ whose edges are the heavy edges and whose vertices are the endpoints of heavy edges.

\begin{claim}\label{clm:triangle-heavy-small}
$|V(F)|\le 6k^2/\eps$ for all sufficiently large $n$.
\end{claim}

\begin{proof}
First, $\Delta(F)\le k-1$. Otherwise, if $v$ is incident with $k$ heavy edges $vu_1,\ldots,vu_k$, then we can greedily choose distinct vertices
\[
        w_i\in N(v)\cap N(u_i)
\]
for $i=1,\ldots,k$, obtaining $k$ triangles $vu_iw_i$ sharing the core $v$, i.e. a copy of $T_{3,k}$.

Second, the matching number of $F$ is at most $3(k+1)/\eps$. Indeed, suppose $F$ contains a matching $x_1y_1,
\ldots,x_my_m$ with $m>3(k+1)/\eps$. Each pair $x_i,y_i$ has at least $\eps n/3$ common neighbors. By averaging, some vertex of $G$ is a common neighbor of at least $k$ of these matched pairs; those $k$ pairs together with this vertex give a copy of $T_{3,k}$.

Taking a maximal matching of $F$, every vertex of $F$ is either in the matching or adjacent to one of its endpoints. Since $\Delta(F)\le k-1$ and the matching size is $O(k/\eps)$, we get $|V(F)|\le 6k^2/\eps$ after adjusting the constant.
\end{proof}

Every triangle of $G$ contains a heavy edge, so $G_0:=G-V(F)$ is triangle-free. Since $|V(F)|=O_{k,\eps}(1)$,
\[
        \delta(G_0)\ge \left(\frac13+\frac{\eps}{2}\right)|G_0|
\]
for all sufficiently large $n$. By Theorem~\ref{thm:KrHomo} with $r=3$, $G_0$ is homomorphic to a triangle-free graph $Q$ of order at most $C_0=C_0(\eps)$. Let
\[
        V(G_0)=Q_1\cup\cdots\cup Q_m
\]
be the corresponding partition into independent classes, with $m\le C_0$.

Refine this partition according to neighborhoods in $V(F)$: two vertices of the same $Q_i$ remain in the same refined part if and only if they have the same neighborhood in $V(F)$. Thus we obtain a bounded partition
\[
        V(G)=V(F)\cup \bigcup_{i=1}^M A_i,
\]
where $M\le C_0 2^{|V(F)|}$, each $A_i$ is independent, and every vertex in $V(F)$ is either complete or anti-complete to each $A_i$. Let $L$ be the quotient graph of this partition, where vertices of $V(F)$ are kept as singleton parts and where two quotient vertices are adjacent if the corresponding two parts contain at least one edge of $G$.

We claim first that $L$ is $T_{3,k}$-free. Suppose not, and let $u$ be the core of a copy of $T_{3,k}$ in $L$.

If $u\in V(F)$, then each blade triangle of the quotient can be realised as a genuine triangle of $G$: singleton vertices are already vertices of $G$, while any non-singleton quotient class consists of vertices with the prescribed adjacency to $V(F)$, and edges between non-singleton classes exist by definition. Greedily realising the $k$ blades gives a copy of $T_{3,k}$ in $G$, a contradiction.

If $u$ is one of the refined parts outside $V(F)$, then the triangle-free nature of $Q$ implies that in each blade triangle of $L$ the two other quotient vertices must both belong to $V(F)$. Since the part corresponding to $u$ is complete to those two singleton vertices, each blade again lifts to a triangle in $G$, and the $k$ lifted triangles share the same chosen vertex from the part $u$. This again gives $T_{3,k}$ in $G$.

Therefore $L$ is $T_{3,k}$-free.

It remains to show that $G$ is the blow-up of $L$. Construct $G'$ from $G$ by completing every pair of refined parts which has at least one edge. By construction, $G'=L[\cdot]$ and $G\subseteq G'$. We claim that $G'$ is still $T_{3,k}$-free. Indeed, any triangle in $G'$ must contain at least two vertices from $V(F)$: outside $V(F)$ the quotient refines a triangle-free graph, and if a triangle had exactly one vertex in $V(F)$, then the definition of the refinement by neighborhoods in $V(F)$ would already realise such a triangle in $G$, contradicting the fact that every triangle of $G$ contains a heavy edge and hence at least two vertices of $V(F)$.

Now any copy of $T_{3,k}$ in $G'$ can be lifted blade by blade to a copy of $T_{3,k}$ in $G$, using the preceding observation. This is impossible. Hence $G'$ is $T_{3,k}$-free. Since $G$ is maximal $T_{3,k}$-free and $G\subseteq G'$, we must have $G=G'=L[\cdot]$.
\end{proof}

\end{document}